\documentclass[11pt]{article}

\usepackage{amsthm}
\usepackage{amsmath}
\usepackage{amssymb}
\usepackage{amsmath, amsthm, bbm}

\usepackage{framed, fullpage}

\usepackage[backref,colorlinks,citecolor=blue,bookmarks=true]{hyperref}

\newtheorem{theo}{Theorem}

\newtheorem{conj}{Conjecture}
\newtheorem{lemma}{Lemma}[section]
\newtheorem{definition}[lemma]{Definition}

\newtheorem{claim}[lemma]{Claim}

\newcommand{\F}{\mathcal{F}}
\renewcommand{\S}{S}

\newcommand{\ith}{i^{th}}
\newcommand{\jth}{j^{th}}

\newcommand{\dth}{d^{th}}

\newcommand{\size}[1]{\left\lvert #1 \right\rvert}
\newcommand{\norm}[1]{\left\lVert #1 \right\rVert}

\newcommand{\WLsym}[1]{\overrightarrow{W}_{#1}}
\newcommand{\SUsym}[1]{S_{#1}}
\newcommand{\WUsym}[1]{W_{#1}}

\newcommand{\SUsymbol}{\SUsym{n}}
\newcommand{\WLsymbol}{\WLsym{n}}
\newcommand{\WUsymbol}{\WUsym{n}}

\newcommand{\WL}{\WLsymbol(p,q)}

\newcommand{\WU}{\WUsymbol(p,q)}

\newcommand{\WLH}{\WLsymbol(p_1,\ldots,p_d)}

\newcommand{\WUH}{\WUsymbol(p_1,\ldots,p_d)}

\newcommand{\floor}[1]{\left\lfloor{#1}\right\rfloor}

\date{}

\title{Exact Bounds for Some Hypergraph Saturation Problems}
\author{
Guy Moshkovitz\thanks{School of Mathematics, Tel-Aviv University, Tel-Aviv, Israel 69978.  Email: {\tt guymosko@tau.ac.il}. Supported in part by ISF grant 224/11.}
\and Asaf Shapira\thanks{School of Mathematics, Tel-Aviv University, Tel-Aviv, Israel 69978, and Schools of Mathematics and Computer Science, Georgia Institute of Technology, Atlanta, GA 30332. Email: {\tt asafico@tau.ac.il}. Supported in part by NSF Grant DMS-0901355, ISF Grant 224/11 and a Marie-Curie CIG Grant 303320.}
}
\begin{document}

\maketitle
\begin{abstract}
Let $W_n(p,q)$ denote the minimum number of edges in an $n \times n$ bipartite graph $G$ on vertex sets $X,Y$ that satisfies the following condition; one can add the edges between $X$ and $Y$ that do not belong to $G$ one after the other so that whenever a new edge is added, a new copy of $K_{p,q}$ is created.
The problem of bounding $W_n(p,q)$, and its natural hypergraph generalization, was introduced by Balogh, Bollob\'as, Morris and Riordan. Their main result, specialized to graphs, used algebraic methods to determine $W_n(1,q)$.

Our main results in this paper give exact bounds for $W_n(p,q)$, its hypergraph analogue, as well as for a new variant of Bollob\'as's Two Families Theorem. In particular, we completely determine $W_n(p,q)$, showing that if $1 \leq p \leq q \leq n$ then
$$
W_n(p,q)=n^2 - (n-p+1)^2 + (q-p)^2\;.
$$
Our proof applies a reduction to a multi-partite version of the Two Families Theorem obtained by Alon. While the reduction is combinatorial, the main idea behind it is algebraic.
\end{abstract}

\section{Introduction}\label{sec:Intro}

One of the most well-known results in Extremal Combinatorics is Bollob\'as's Two Families Theorem~\cite{Bollobas65}, which states\footnote{This formulation, which is equivalent to the result of Bollob\'as, was actually conjectured by Ehrenfeucht and Mycielski, and confirmed by Katona~\cite{Katona74}, Jaeger and Payan~\cite{JaegerPa71}, and Tarj\'an~\cite{Tarjan75}.} that if $A_1,\ldots,A_h$ and $B_1,\ldots,B_h$ are two families of sets satisfying $A_i\cap B_j=\emptyset$ if and only if $i=j$, and if $\size{A_i}\le a$ and $\size{B_i}\le b$, then $h\le \binom{a+b}{b}$.
While this theorem has many applications,
Bollob\'as's motivation for proving it was an extremal graph/hypergraph  {\em saturation} problem, which we further discuss in Subsection \ref{ssec:Hypergraphs}.

Our most general result in this paper is a variant of the Two Families Theorem for multi-partite sets.
Our main motivation for proving this result was a problem considered by
Balogh, Bollob\'as, Morris and Riordan~\cite{BaloghBoMoRi11} which is a variant of the saturation problem originally studied by Bollob\'as~\cite{Bollobas65}.
We elaborate on this aspect of the paper in Subsection \ref{ssec:main}. As we show later in the paper, our new Two Families Theorem can be used to fully resolve the problem considered in~\cite{BaloghBoMoRi11}.

To state our new version of the Two Families Theorem we need the following definition, where here and throughout the paper, we use $[n]$ to denote the set $\{1,\ldots,n\}$.

\begin{definition}\label{def:Q}
Let $a_1,\ldots,a_d$ and $b_1,\ldots,b_d$ be nonnegative integers with $a:=\max_i a_i$, and take $U_1,\ldots,U_d$ to be disjoint sets, where $|U_i|=a+b_i$, and (with a slight abuse of notation) we think of each set $U_i$ as $[a+b_i]$.
We define $Q(a_1,\ldots,a_d,b_1,\ldots,b_d)$ as the number of
sets $S\subseteq U_1\cup\cdots\cup U_d$ for which there is a permutation $\pi:[d]\to[d]$ so that $S\cap U_i$ is a subset of $[a_{\pi(i)}+b_i]$ of size $b_i$.
\end{definition}

Our new Two Families Theorem is as follows.

\begin{theo}\label{thm:Combin}
Let $X_1,\ldots,X_d$ be $d$ disjoint sets and let $a_1,\ldots,a_d,b_1,\ldots,b_d$ be nonnegative integers.
Suppose $A_1,\ldots,A_h$ and $B_1,\ldots,B_h$ are two families of subsets of $X_1\cup\cdots\cup X_d$ that satisfy:
\begin{enumerate}
\item $A_i\cap B_i=\emptyset$ for every $1 \leq i \leq h$.
\item $A_i\cap B_j\neq\emptyset$ for every $1\leq  i<j \leq h$.
\item For every $1\le i\leq h$ and $1\le j\le d$ we have $\size{B_i\cap X_j}\leq b_j$.
\item For every $1\le i\leq h$ there is a permutation $\pi:[d]\to[d]$ so that for every $1\le j\le d$ we have $\size{A_i\cap X_j}\leq a_{\pi(j)}$.
\end{enumerate}
Then $$h\le Q(a_1,\ldots,a_d,b_1,\ldots,b_d)\;.$$
Moreover, this bound is best possible for any choice of $a_1,\ldots,a_d$ and $b_1,\ldots,b_d$.
\end{theo}

We refer the reader to the end of Section~\ref{sec:thm_combin} for an explicit formula for $Q(a_1,\ldots,a_d,b_1,\ldots,b_d)$.
Variants and special cases of Theorem~\ref{thm:Combin} were proved by several authors. Most notably, Alon~\cite{alon}, using techniques from exterior algebra, proved a Two Families Theorem that differs from Theorem~\ref{thm:Combin} in that the permutation $\pi$ is restricted to be the identity permutation (i.e., the fourth condition is replaced by the requirement that for every $1\le i\leq h$, $1\le j\le d$ we have $\size{A_i\cap X_j}\leq a_j$).
Alon's theorem then states that $h\le \prod_{i=1}^d \binom{a_i+b_i}{b_i}$.
Alon's theorem was preceded by a proof of the special case ${d=1}$, which is often called the \emph{skew} Two Families\footnote{Strictly speaking, the conditions in the skew version assume an ordering of the sets, so the two families are really two sequences; however, we keep the term ``Two Families'' for the sake of consistency.} theorem. This theorem was proved by Lov\'asz~\cite{Lovasz}, Frankl~\cite{Fr} and Kalai~\cite{Kalai} in some of the classical applications of the linear algebra method in Combinatorics. Interestingly, finding a combinatorial proof for it is still open.

\subsection{Background on saturation problems}\label{ssec:Hypergraphs}

While our main motivation in this paper is a certain saturation problem in bipartite graphs (and more generally, $d$-uniform $d$-partite hypergraphs),
we begin by mentioning some classical results on saturation problems in non-bipartite graphs.
A graph $G$ is \emph{strongly saturated} with respect to a graph $H$ (or strongly $H$-saturated) if $G$ does not contain a copy of $H$, yet adding any new edge to $G$ creates a copy of $H$. The problem of strong saturation asks for the \emph{minimum} number of edges in an $n$-vertex graph that is strongly $H$-saturated, for different graphs $H$ of interest (notice that the ``dual'' problem, of finding the \emph{maximum} number of edges in an $n$-vertex $H$-saturated graph, is of course the classical Tur\'an problem).
Let $\SUsymbol(p)$ be the minimum number of edges in an $n$-vertex graph that is strongly $K_p$-saturated, where $2\le p\le n$.
The problem of determining $\SUsymbol(p)$ was considered already in the 1940's by Zykov~\cite{Zykov49}, and later by Erd{\H{o}}s, Hajnal and Moon~\cite{ErdosHaMo64} who showed that $\SUsymbol(p)=\binom{n}{2}-\binom{n-p+2}{2}$.
The upper bound on $\SUsymbol(p)$ is easy, as removing the edges of a $K_{n-p+2}$ from $K_n$ clearly gives a strongly $K_p$-saturated graph.
Bollob\'as's Two Families Theorem, mentioned at the beginning of this paper, gives a tight lower bound for $S_n(p)$ and for its natural hypergraph generalization (see the end of Section~\ref{sec:thm_combin} for a similar reduction).

The following notion of saturation was originally introduced by Bollob\'as~\cite{Bela68}.
A graph $G$ is \emph{weakly saturated} with respect to a graph $H$ (or weakly $H$-saturated) if all the non-edges of $G$ can be added one at a time, in some order, so that each new edge creates a new copy of $H$.
We refer to the corresponding ordering of the non-edges of $G$ as a \emph{saturation process} of $G$ with respect to $H$.
For example, it is not hard to see that the weakly $K_3$-saturated graphs with the minimum number of edges are precisely the trees; notice that already for $K_3$ the extremal examples are not unique, suggesting that the general problem might be quite challenging.

Let $\WUsymbol(p)$ be the minimum number of edges in an $n$-vertex graph that is weakly $K_p$-saturated, where $2\le p\le n$.
Notice that for any $H$, a strongly $H$-saturated graph is in particular weakly $H$-saturated, so $\WUsymbol(p)\le \SUsymbol(p)$.
It follows from the skew version of the Two Families Theorem that in fact $\WUsymbol(p)=\SUsymbol(p)$. That being said, the extremal graphs are not the same; there are weakly $K_p$-saturated graphs with $\binom{n}{2}-\binom{n-p+2}{2}$ edges which are not strongly $K_p$-saturated.

\subsection{Weak saturation in multi-partite hypergraphs}\label{ssec:main}

In this paper we focus on saturation problems in the setting of bipartite graphs, and more generally, $d$-uniform $d$-partite hypergraphs. This variant of the problem was first introduced in 1964 by Erd{\H{o}}s, Hajnal and Moon~\cite{ErdosHaMo64}.
Unlike the definition of saturation in the previous subsection, here (and henceforth) the only edges that are considered are those containing one vertex from each vertex class.
Let $H$ be a $d$-uniform $d$-partite hypergraph with vertex classes $V_1,\ldots,V_d$.
We say that $H$ is weakly $K^d_{p_1,\ldots,p_d}$-saturated\footnote{$K^d_{p_1,\ldots,p_d}$ denotes the complete $d$-uniform $d$-partite hypergraph with vertex classes of sizes $p_1,\ldots,p_d$.} if all edges---containing one vertex from each $V_i$---that do not belong to $H$ can be added to $H$ one after the other so that whenever a new edge is added, a new copy of $K^d_{p_1,\ldots,p_d}$ is created.\footnote{Saturation in the setting of $d$-partite hypergraphs is referred to in some papers as $d$-saturation, or bi-saturation if $d=2$.
Since we henceforth only consider saturation in this setting, we prefer to keep using the term ``saturation''.}

Our main motivation in this paper is the question of determining the following function.
For integers ${1\le p_1,\ldots,p_d \le n}$, let $\WUH$ be the smallest number of edges in a $d$-uniform $d$-partite hypergraph, with $n$ vertices in each vertex class, that is weakly $K^d_{p_1,\ldots,p_d}$-saturated.
While this notion of weak saturation in $d$-uniform $d$-partite hypergraphs was introduced only recently by Balogh et al.~\cite{BaloghBoMoRi11}, a similar notion of weak saturation, in which the copies of $K^d_{p_1,\ldots,p_d}$ are required to have $p_i$ vertices in the $\ith$ vertex class, was considered long before; we refer to this notion as \emph{directed} weak saturation.
For integers ${1\le p_1,\ldots,p_d \le n}$ denote $\WLH$ the directed analogue of $\WUH$.
Alon~\cite{alon}
determined $\WLH$ exactly, showing that $\WLH= n^d-\prod_{i=1}^d (n-p_i+1)$.
Note that, by definition, $\WUsymbol(p,\ldots,p)=\WLsymbol(p,\ldots,p)$, and so we can deduce from Alon's result a partial answer to the question considered in this paper, namely,
\begin{equation}\label{eq:symmetricH}
\WUsymbol(p,\ldots,p) = n^d-(n-p+1)^d \;.
\end{equation}
A partial answer for a different setting of parameters was given by Balogh et al.~\cite{BaloghBoMoRi11}. Their main result, proved using linear algebraic techniques, determined $\WUH$ when the $p_i$ only take the two values $1$ and $q$, for some positive integer $q$.

In this paper we determine $\WUH$ for all values of $p_1,\ldots,p_d$ and $n$.
To state our result we need the following definition.
\begin{definition}\label{def:q_n}
For integers ${1\le p_1\le\cdots\le p_d\le n}$,
let $q_n(p_1,\ldots,p_d)$ be the number of $d$-tuples $x\in[n]^d$ such that $x_{(i)} \ge p_i$ for every $1\le i\le d$, where $x_{(i)}$ is the $\ith$ smallest element in the sorted $d$-tuple of $x$ (i.e., which includes repetitions).\footnote{For example, if $x=(5,2,5,1)$ then the sorted $4$-tuple of $x$ is $(1,2,5,5)$.}
\end{definition}

Our main theorem is as follows.

\begin{theo}\label{thm:Main}
For all integers $1 \le p_1\le \cdots \le p_d\le n$ we have
$$\WUH=n^d - q_n(p_1,\ldots,p_d)\;.$$
\end{theo}

It is of course interesting to find explicit formulas for $q_n(p_1,\ldots,p_d)$, and thus for $\WUH$; we do so in Section~\ref{sec:thm_main}.
By combining Theorem~\ref{thm:Main} and the explicit formulas, we obtain the interesting corollary that if $p_d=o(n)$ then $\WUH$ is asymptotically determined only by $p_1$, namely,
\begin{equation}\label{eq:W_asymptotic}
\WUH = (d(p_1-1)+o(1))n^{d-1} \;.
\end{equation}
This should be compared with the fact that the directed analogue $\WLH$ is asymptotically determined by all $p_1,\ldots,p_d$; specifically, if $p_1,\ldots,p_d$ are all of order $o(n)$
then $\WLH=(p_1+\ldots+p_d-d+o(1))n^{d-1}$.

\subsection{Proof overview}

Let us finally remark on the proofs of the theorems stated above.
Interestingly, our proof of Theorem~\ref{thm:Combin} proceeds by an indirect argument that reduces Theorem~\ref{thm:Combin} to Alon's Two Families Theorem (see Theorem~\ref{thm:Noga} for the exact statement). Actually, our proof proceeds by reducing an instance with
$a_1,\ldots,a_d$ to one where all $a_i$ are replaced by $\max_ia_i$. This might seem counter-intuitive since enlarging the $a_i$ {\em increases} the upper bound
in Theorem \ref{thm:Noga}. The catch is that when we will come to apply the bound for Theorem~\ref{thm:Noga}, the fact that we have increased the $a_i$ is going to allow us to add some ``dummy'' pairs of sets $A'_i,B'_i$ so that the new instance will still satisfy the requirements of Theorem \ref{thm:Noga}.
Somehow the trade-off between increasing $a_1,\ldots,a_d$ and adding the dummy sets results in a tight bound.
We prove Theorem~\ref{thm:Main} using a reduction along similar lines.

The alert reader has probably noticed that the ``trick'' we use here, namely, adding extra sets to the families before applying the upper bound, is somewhat reminiscent of the trick used by Blokhuis to improve the bound on the size of $2$-distance sets in Euclidean space~\cite{Blokhuis84}. And indeed, our original proof was a direct one, applying the algebraic proof of Alon's theorem by Blokhuis~\cite{Blokhuis90}, via resultants of polynomial, together with the trick from~\cite{Blokhuis84} of adding extra polynomials in order to improve the upper bound. As it happens, we later realized that it is in fact possible to reduce the problem to Alon's theorem---no algebraic machinery necessary! (at least not explicitly)

\paragraph*{Organization:}
The rest of the paper is organized as follows.
We prove Theorem~\ref{thm:Main} in Section~\ref{sec:thm_main}.
Section~\ref{sec:thm_main} also contains some explicit formulas for $\WUH$.
In Section~\ref{sec:thm_combin} we prove our new Two Families Theorem, Theorem~\ref{thm:Combin}. We also give an explicit formula for $Q(a_1,\ldots,a_d,b_1,\ldots,b_d)$, as well as briefly
explain how can one obtain an alternative proof of Theorem~\ref{thm:Main} using Theorem~\ref{thm:Combin}.
Section \ref{sec:conclude} contains some concluding remarks and open problems.

\section{Undirected Weak Saturation of Hypergraphs}\label{sec:thm_main}

We begin by proving the upper bound in Theorem~\ref{thm:Main}.
As a warm-up, let us briefly describe the construction proving the upper bound  for the special case of graphs.
Specifically, we show that $\WU \le n^2 - (n-p+1)^2 + (q-p)^2$.
Consider the $n\times n$ bipartite graph with vertex classes $\{x_1,\ldots,x_n\}$ and $\{y_1,\ldots,y_n\}$ which is the union of three complete graphs: two $K_{p-1,n}$'s, one with edges $\{x_iy_j : i<p\}$ and the other with edges  $\{x_iy_j : j<p\}$, and a $K_{q-p,q-p}$ with edges $\{x_iy_j : {p \le i,j < q}\}$.
To see that this graph is weakly $K_{p,q}$-saturated, simply add the edges $x_iy_j$ with $i,j\ge q$ only after the rest of the missing edges are added.
The reader may easily verify that by adding the missing edges in this order, a new copy of $K_{p,q}$ is indeed created upon each addition.

We now generalize the construction above to the case of hypergraphs.
Let $G_0=G_0(p_1,\ldots,p_d)$ be the $d$-uniform $d$-partite hypergraph whose non-edges are enumerated by $q_n(p_1,\ldots,p_d)$ (recall Definition~\ref{def:q_n}). More formally, let the vertex classes $V_1,\ldots,V_d$ of $G_0$ each contain $n$ vertices, and let us label the vertices in each set by $1,2,\ldots,n$ (abusing notation slightly).
Let us henceforth identify edges with $d$-tuples in $[n]^d$.\footnote{I.e., the $d$-tuple $(x_1,\ldots,x_d)\in[n]^d$ is identified with the edge containing from each $V_i$ the vertex labeled $x_i$.}
Then an edge $(x_1,\ldots,x_d)\in V_1 \times \cdots \times V_d$ {\bf does not} belong to $G_0$ if and only if for every $1\le i\le d$, the $\ith$ smallest element (i.e., when $x_1,\ldots,x_d$ are sorted with repetitions) is at least $p_i$.
We now show that $G_0$ is weakly $K^d_{p_1,\ldots,p_d}$-saturated, proving the upper bound in Theorem~\ref{thm:Main}.
Henceforth we use $\norm{H}$ for the number of edges in a hypergraph $H$.

\begin{lemma}\label{lemma:main_upper}
$\WUH \leq \norm{G_0}$.
\end{lemma}
\begin{proof}
Call $x_1+\cdots+x_d$ the {\em weight} of the edge $e=(x_1,\ldots,x_d)$, and denote $G_w$ the $d$-uniform $d$-partite hypergraph obtained from $G_0$ by adding every edge of weight at most $w$.
We next prove that adding any new edge of weight $w$ to $G_{w-1}$ creates a new copy of $K^d_{p_1,\ldots,p_d}$.
From this it clearly follows by induction on $w$ that $G_0$ is weakly $K^d_{p_1,\ldots,p_d}$-saturated, as required.

Let $e=(x_1,\ldots,x_d)$ be an edge of weight $w$ and suppose $e$ is not in $G_{w-1}$.
We next construct for each $1\le i\le d$ a set $S_i$ of vertices from the $\ith$ vertex class. Fix $1\le i\le d$, and suppose that $x_i$ is the $\jth$ smallest among $x_1,\ldots,x_d$ (i.e., when ordered with repetitions).
We let $S_i$ be the set of vertices, from the $\ith$ vertex class, labeled by $1,2,\ldots,p_j-1$ (recall $p_1\le\cdots\le p_d$).
Since $e$ is not in $G_w$ and hence not in $G_0$, it follows from the definition of $G_0$ that $x_i \ge p_j$.
Therefore, $S_i\cup\{x_i\}$ has $p_j$ (distinct) elements.
Note that every edge spanned by $\bigcup_{i=1}^d (S_i\cup\{x_i\})$, except for $e$, is of weight smaller than that of $e$, and so is contained in $G_{w-1}$.
This means that adding $e$ to $G_0$ creates a new copy of $K^d_{p_1,\ldots,p_d}$ spanned by the vertices $\bigcup_{i=1}^d (S_i\cup\{x_i\})$,
thus completing the proof.
\end{proof}

We next turn to the proof of the lower bound in Theorem~\ref{thm:Main}.
Let us start with a quick argument showing that in the graph case we have $\WU \ge n^2 - (n-p+1)^2 + (q-p)^2$.
Given a weakly $K_{p,q}$-saturated bipartite graph, add $q-p$ new vertices to each vertex class, connecting each new vertex to all the original vertices in the other class. A moment's thought reveals that this new graph is weakly $K_{q,q}$-saturated. This means that the number of edges in the new graph is at least $W_{n+q-p}(q,q)$, and applying~(\ref{eq:symmetricH}) we get the desired lower bound.

We will now show how one can use the hypergraph $G_0$ we constructed earlier to prove that every weakly $K^d_{p_1,\ldots,p_d}$-saturated hypergraph must have as many edges as $G_0$.
First, we will need to use the property of $G_0$ that its complement,\footnote{By complement we mean relative to $K_{n,\ldots,n}$, that is, the hypergraph that contains an edge $(x_1,\ldots,x_d)\in V_1 \times \cdots \times V_d$ if and only if $G_0$ does not.} denoted $\overline{G_0}$, contains every possible ``orientation'' of $K^d_{n-p_1+1,\ldots,n-p_d+1}$.

\begin{claim}\label{claim:complement}
For every permutation $\pi:[d]\to[d]$, the hypergraph $\overline{G_0}$ contains a copy of the hypergraph $K^d_{n-p_1+1,\ldots,n-p_d+1}$ having $n-p_{\pi(i)}+1$ vertices in the $\ith$ vertex class.
\end{claim}
\begin{proof}
We start with a simple observation, claiming that if two tuples of real numbers $x=(x_1,\ldots,x_d)$ and $y=(y_1,\ldots,y_d)$ satisfy $x_i\ge y_i$ for every $1\le i\le d$, then they satisfy $x_{(i)}\ge y_{(i)}$ for every $1\le i\le d$ as well (where, as usual, $x_{(i)}$ is the $\ith$ smallest element in the sorted tuple of $x$, and similarly for $y$).
To see this, let $\sigma:[d]\to[d]$ be a permutation sorting $y$, that is, $y_{\sigma(1)}\le\cdots\le y_{\sigma(d)}$.
Now note that for every $1\le i\le d$ and $i\le j\le d$ we have $x_{\sigma(j)} \ge y_{\sigma(j)} \ge y_{\sigma(i)} = y_{(i)}$.
This means that $x$ has at least $d-i+1$ elements that are at least as large as $y_{(i)}$, which means that we must have $x_{(i)}\ge y_{(i)}$.

Now, suppose without loss of generality that $p_1\le\cdots\le p_d$.
Let $\pi:[d]\to[d]$ be an arbitrary permutation and let $S_i$ be the subset of vertices of the $\ith$ vertex class containing those vertices labeled by $p_{\pi(i)},p_{\pi(i)}+1,\ldots,n$.
Then for every edge $e=(x_1,\ldots,x_d)$ spanned by the vertices in $\bigcup_{i=1}^d S_i$ it holds that $x_i\ge p_{\pi(i)}$.
It now follows from our observation above that $x_{(i)}\ge p_i$.
By the definitions of $q_n(p_1,\ldots,p_d)$ and $G_0$ we conclude that $e\in\overline{G_0}$.
Hence, $\bigcup_{i=1}^d S_i$ spans a copy of $K^d_{n-p_1+1,\ldots,n-p_d+1}$ in $\overline{G_0}$ having $n-p_{\pi(i)}+1$ vertices in the $\ith$ vertex class, as desired.
\end{proof}

\begin{lemma}\label{lemma:main_lower}
$\WUH \geq \norm{G_0}$.
\end{lemma}
\begin{proof}
Let $H$ be a $d$-uniform $d$-partite hypergraph that is weakly $K^d_{p_1,\ldots,p_d}$-saturated, where its vertex classes $(V_1,\ldots,V_d)$ are each of cardinality $n$.
We construct a hypergraph $H'$ with $2n$ vertices in each vertex class by combining it with $\overline{G_0}$
as follows.
Let $U_1,\ldots,U_d$ be $d$ sets of new vertices (i.e., disjoint from $\bigcup_{i=1}^d V_i$ and from each other) with $\size{U_i}=n$.
We let $H'$ be the $d$-uniform $d$-partite hypergraph with vertex classes $(V_1\cup U_1,\ldots,V_d\cup U_d)$ whose edges are defined as follows.
The edges of $H'$ that are spanned by the vertices in $V:=\bigcup_{i=1}^d V_i$ are precisely those of $H$; the edges of $H'$ that are spanned by the
vertices in $U:=\bigcup_{i=1}^d U_i$ are precisely those of $\overline{G_0}$; finally, all other possible edges (i.e., those containing at least one vertex from $V$ and at least one vertex from $U$) appear in $H'$ as well.
Notice that by counting the non-edges of $H'$ we get
\begin{align}\label{eq:H'}
(2n)^d-\norm{H'}=(n^d-\norm{H})+\norm{G_0} \;.
\end{align}
We claim that $H'$ is weakly $K^d_{n+1,\ldots,n+1}$-saturated.
Observe that~(\ref{eq:symmetricH}) and~(\ref{eq:H'}) would then give
$$
(2n)^d-(n^d-\norm{H})-\norm{G_0}=\norm{H'}\geq \WUsym{2n}(n+1,\ldots,n+1)=(2n)^d-n^d\;,
$$
implying that $\norm{H} \ge \norm{G_0}$, thus completing the proof.

To show that $H'$ is weakly $K^d_{n+1,\ldots,n+1}$-saturated, we claim that one obtains a saturation process of $H'$ with respect to $K^d_{n+1,\ldots,n+1}$ by first adding the non-edges of $H$ in the same order they appear in some saturation process of $H$ (with respect to $K^d_{p_1,\ldots,p_d}$), and then adding, in an arbitrary order, all edges of $G_0$.
To see that this indeed defines a saturation process of $H'$ with respect to $K^d_{n+1,\ldots,n+1}$, let $e$ be a non-edge of $H$ added at some point.
Then adding $e$ to $H'$ (after all the edges that precede $e$ in the saturation process are added) creates a new copy of $K^d_{p_1,\ldots,p_d}$ in $H'$, which we denote $C$.
Let $\pi:[d]\to[d]$ be a permutation such that $C$ contains $p_{\pi(i)}$ vertices in the $\ith$ vertex class for every $1\le i\le d$.
By Claim~\ref{claim:complement}, $\overline{G_0}$ contains a copy $C'$ of $K^d_{n-p_1+1,\ldots,n-p_d+1}$ having $n-p_{\pi(i)}+1$ vertices in the $\ith$ vertex class. It follows that when adding $e$ we in fact create a new copy of $K^d_{n+1,\ldots,n+1}$ in $H'$, namely, the copy spanned by the union of the vertex sets of $C$ and $C'$.
To complete the proof of our claim we observe that, after all the edges over $V$ are added to $H'$, each edge $(x_1,\ldots,x_d)$ of $G_0$ is the only missing edge in the copy of $K^d_{n+1,\ldots,n+1}$ spanned by $\bigcup_{i=1}^d \left(V_i\cup\{x_i\}\right)$ (recall that $\size{V_i}=n$ for every $i$).
This completes the proof of the statement.
\end{proof}

\begin{proof}[Proof of Theorem~\ref{thm:Main}]
Lemmas~\ref{lemma:main_upper} and~\ref{lemma:main_lower} give $\WUH = \norm{G_0} = n^d - q_n(p_1,\ldots,p_d)$.
\end{proof}

\subsection{Explicit formulas for $\WUH$}

We begin by computing $q_n(p_1,\ldots,p_d)$ in some easy special cases. When the $p_i$ take only one value we clearly have $q_n(v,\ldots,v)=(n-v+1)^d$.
When the $p_i$ take two values $v_1\le v_2$, where $v_1$ occurs $r$ times, it is easy to see that $x=(x_1,\ldots,x_d)\in[n]^d$ is enumerated by $q_n$ if and only if it holds that $x_i\ge v_1$ for all $i$ and the number of $x_i$ smaller than $v_2$ is at most $r$; thus,
$q_n(v_1,\ldots,v_1,v_2,\ldots,v_2)=\sum_{i=0}^{r}\binom{d}{i} (v_2-v_1)^{i}(n-v_2+1)^{d-i}$.

Let us consider the general case. Suppose $p_1,\ldots,p_d$ take $m+1$ distinct values $v_1<\cdots<v_{m+1}$, where $v_i$ occurs $r_i$ times, $1\le i\le m+1$ (so $r_{m+1}=d-r_1-\cdots-r_m$).
For $x=(x_1,\ldots,x_d)\in[n]^d$ set $i_j=\size{\{i : v_j \le x_i < v_{j+1}\}}$.
Then it is not hard to see that $x$ is enumerated by $q_n$ if and only if $x_i\ge v_1$ for all $i$ and moreover $i_1 \le r_1 ,\, i_1+i_2 \le r_1+r_2 ,\,\ldots,\,i_1+\cdots+i_m\le r_1+\cdots+r_m$.
This gives the following explicit formula;
$$q_n(p_1,\ldots,p_d) = \sum_{i_1,\ldots,i_m}\binom{d}{i_1,\ldots,i_m}\prod_{j=1}^{m} (v_{j+1}-v_{j})^{i_j} \cdot (n-v_{m+1}+1)^{d-\sum_{k=1}^m i_k}$$
where the sum is over all $i_1,\ldots,i_m$ satisfying, for every $1\le j\le m$, the inequality $i_1+\cdots+i_j \leq r_1+\cdots+r_j$.\footnote{The notation $\binom{n}{k_1,\ldots,k_m}$ stands for the multinomial coefficient, that is, $n!/(k_1!\cdots k_m!\ell!)$ where $\ell=n-\sum_{j=1}^m k_i$.}
We note that when all the $p_i$ are distinct, that is, when $r_1=\cdots=r_d=1$, the number of summands in the above formula is the $\dth$ Catalan number. So in a sense, $q_n(p_1,\ldots,p_d)$ may be thought of as a ``weighted'' Catalan number.

An alternative description for $\WUH$ can be obtained as follows.
By Definition~\ref{def:q_n} and Theorem~\ref{thm:Main} we have that $\WUH$ equals the number of $d$-tuples $x\in[n]^d$ such that $x_{(i)} < p_i$ holds for at least one $1\le i\le d$.
Consider now the set of $d$-tuples
$$L_i(t)=\Big\{(x_1,\ldots,x_d)\in[n]^d \,:\, \size{\{j:x_j < t\}} = i\Big\}.$$
A moment's thought reveals that\footnote{Indeed, if $i$ is the largest such that $x_{(i)}<p_i$ then clearly $x\in L_i(p_i)$; conversely, if $x\in L_i(p_i)$ then $x_{(i)}<p_i$.}
\begin{equation}\label{eq:incexc0}
\WUH=\size{\bigcup_{i=1}^d L_i(p_i)}.
\end{equation}
We therefore obtain the inclusion-exclusion formula
\begin{equation}\label{eq:incexc}
\WUH = \sum_{\emptyset\neq I\subseteq [d]}
(-1)^{\size{I}+1}\size{\bigcap_{i\in I} L_i(p_i)}\;.
\end{equation}
It is easy to see that
\begin{equation}\label{eq:incexc1}
\size{L_i(p_i)}=\binom{d}{i}(p_i-1)^i(n-p_i+1)^{d-i},
\end{equation}
and that for arbitrary $I=\{i_1<i_2<\cdots<i_t\}$ we have
\begin{equation}\label{eq:incexc2}
\size{\bigcap_{i\in I} L_i(p_i)} =
\binom{d}{i_1,\ldots,i_t}(p_{i_1}-1)^{i_1} \cdot \prod_{j=2}^t (p_{i_j}-p_{i_{j-1}})^{i_j-i_{j-1}} \cdot (n-p_{i_t}+1)^{d-\sum_{k=1}^t i_k} \;.
\end{equation}
Plugging (\ref{eq:incexc2}) into (\ref{eq:incexc}) we get another explicit formula for $\WUH$.

Note that (\ref{eq:incexc0}) gives the crude bound $\size{L_1(p_1)} \le \WUH \le \sum_{i=1}^d \size{L_i(p_i)}$.
Combining it with (\ref{eq:incexc1}) we get that
$$
d(p_1-1)(n-p_1+1)^{d-1} \le \WUH \le \sum_{i=1}^d \binom{d}{i}(p_i-1)^i(n-p_i+1)^{d-i}\;,
$$
which implies the asymptotic formula~(\ref{eq:W_asymptotic}) mentioned in the Introduction.

\section{Undirected Two Families Theorem}\label{sec:thm_combin}

In this section we prove Theorem~\ref{thm:Combin}, give an explicit formula for $Q(a_1,\ldots,a_d,b_1,\ldots,b_d)$, as well as give an alternative proof of Theorem~\ref{thm:Main}.
Our proof of Theorem~\ref{thm:Combin} will follow by a reduction to Alon's Two Families Theorem which we repeat here.

\begin{theo}[\textbf{Alon~\cite{alon}}]\label{thm:Noga}
Let $X_1,\ldots,X_d$ be $d$ disjoint sets and let $a_1,\ldots,a_d,b_1,\ldots,b_d$ be nonnegative\footnote{The original statement in~\cite{alon} considers positive integers, but it is easy to see that it implies the statement with nonnegative integers.} integers.
Suppose $A_1,\ldots,A_h$ and $B_1,\ldots,B_h$ are two families of subsets of $X_1\cup\cdots\cup X_d$ satisfying:
\begin{enumerate}
\item $A_i\cap B_i=\emptyset$ for every $1 \leq i \leq h$.
\item $A_i\cap B_j\neq\emptyset$ for every $1\leq  i<j \leq h$.
\item For every $1\le i\leq h$ and $1\le j\le d$ we have $\size{A_i\cap X_j}\leq a_j$ and $\size{B_i\cap X_j}\leq b_j$.

\end{enumerate}
Then $h\le \prod_{j=1}^d \binom{a_j+b_j}{b_j}$.
\end{theo}

Recall that the conditions of Theorem~\ref{thm:Combin} differ from those in  Alon's theorem only in that $\size{A_i\cap X_j}\leq a_j$ need not hold, and instead it is only required that for every $1\le i\le h$ there is a permutation $\pi:[d]\to[d]$ so that $\size{A_i\cap X_j}\leq a_{\pi(j)}$ for $1\le j\leq d$.

\subsection{Proof of Theorem~\ref{thm:Combin}}

We start with the proof of the upper bound.

\begin{proof}
Put $a=\max_i a_i$. Let $U_1,\ldots,U_d$ be $d$ mutually disjoint sets,
where $U:=\bigcup_{j=1}^d U_j$ is also disjoint from $X_1\cup\cdots\cup X_d$,
and where each $U_j$ contains $a+b_j$ elements labeled $1,2,\ldots,a+b_j$.
Setting $X'_i:=X_i\cup U_i$, the idea is to transform any two families over $X_1\cup\cdots\cup X_d$ satisfying the conditions of Theorem~\ref{thm:Combin}  into two larger families over $X'_1\cup\cdots\cup X'_d$ satisfying the conditions of Theorem~\ref{thm:Noga}.

For a permutation $\pi:[d]\to[d]$ we denote $C_{\pi}$ the subset of $U$ containing from each $U_j$ its last $a-a_{\pi(j)}$ members, that is, those labeled by $a_{\pi(j)}+b_j+1,\ldots,a+b_j$.
For each $1\le i\le h$ fix a permutation $\pi=\pi_i:[d]\to[d]$ satisfying $\size{A_i\cap X_j}\le a_{\pi(j)}$ (one exists by the fourth condition in the statement in Theorem~\ref{thm:Combin}) and set $A'_i=A_i\cup C_{\pi}$.
Note that $\size{A'_i\cap X'_j} \le a$ for every $1\le j\le d$.
It is thus clear that the two families $A'_1,\ldots,A'_h$ and $B_1,\ldots,B_h$ satisfy the first and second conditions in the statement of  Theorem~\ref{thm:Noga} with respect to $X'_1,\ldots,X'_d$, simply because every $C_{\pi}$ is disjoint from $X_1\cup\cdots\cup X_d$; furthermore, they also satisfy the third condition with ${a_1=\cdots=a_d=a}$ and the same $b_1,\ldots,b_d$.

Now, suppose that we are able to add $h'$ new sets to each family---with the sets in the first family (i.e., the $A_i$'s in the statement) containing at most $a$ elements from each part $X'_j$ and the sets in the second family (i.e., the $B_i$'s) containing at most $b_j$ elements from $X'_j$---while still satisfying the first and second conditions of Theorem~\ref{thm:Noga}.
Applying Theorem~\ref{thm:Noga} would then yield the upper bound $h+h'\le \prod_{j=1}^d \binom{a+b_j}{b_j}$.
Therefore, to complete the proof it suffices to show that we may extend the two families by $h'$ new sets where
$h' = \prod_{j=1}^d \binom{a+b_j}{b_j}-Q(a_1,\ldots,a_d,b_1,\ldots,b_d)$.

Note that $Q(a_1,\ldots,a_d,b_1,\ldots,b_d)$ equals the number of sets $B\subseteq U$, with $\size{B\cap U_j}=b_j$, for which there is a permutation $\pi:[d]\to[d]$ so that $B\cap C_{\pi}=\emptyset$.
Thus, $h'=\prod_{j=1}^d \binom{a+b_j}{b_j}-Q(a_1,\ldots,a_d,b_1,\ldots,b_d)$ is the number of sets $B\subseteq U$, with $\size{B\cap U_j}=b_j$, such that for every permutation $\pi$, $B\cap C_{\pi}\neq\emptyset$.
Let $B'_1,\ldots,B'_{h'}$ denote the sets enumerated by $h'$, and consider the two families\footnote{We write $\overline{S}$ for the complement of $S$ in $U$, that is, $U\setminus S$.}
$$A'_1,\ldots,A'_h,\overline{B'_1},\ldots,\overline{B'_{h'}}$$
and
$$B_1,\ldots,B_h,B'_1,\ldots,B'_{h'} \;.$$
Since any $B'_i$ contains $b_j$ elements from each $U_j$, we have that $\overline{B'_i}$ contains $a$ elements from each $U_j$.
Therefore, every set in the first family contains at most $a$ members from each part $X'_j$, and every set in the second family contains at most $b_j$ members from each part $X'_j$, as desired.
We claim that the above two families satisfy the conditions of   Theorem~\ref{thm:Noga}, which would complete the proof by applying that theorem as discussed above.
The first condition in the statement is clearly satisfied, as $A'_i\cap B_i=\emptyset$ and $\overline{B'_i}\cap B'_i=\emptyset$. As for the second condition, recall that, as observed above, $A'_i\cap B_j\neq\emptyset$ when $i\neq j$; moreover, it is clear that for any $i\neq j$ we have $\overline{B'_i}\cap B'_j\neq\emptyset$, as $B'_j\nsubseteq B'_i$.
It remains to show that for every $1\le i\le h$ and every $1\le j\le h'$ we have $A'_i\cap B'_j\neq\emptyset$.
Indeed, any $A'_i$ contains some $C_{\pi}$ and any $B'_j$ intersects every $C_{\pi}$. This completes the proof.
\end{proof}

We now show that the bound in Theorem~\ref{thm:Combin} is best possible for any choice of $a_1,\ldots,a_d$ and $b_1,\ldots,b_d$.

\begin{proof}
Given $a_1,\ldots,a_d$ and $b_1,\ldots,b_d$ let $h=Q(a_1,\ldots,a_d,b_1,\ldots,b_d)$. We need to construct two families of $h$ sets satisfying the four conditions of Theorem~\ref{thm:Combin}. For a set $B\subseteq U_1\cup\cdots\cup U_d$ (where $U_1,\ldots,U_d$ are as in the definition of $Q$), let $w(B)$ be the sum of the labels of its members, that is, $w(B)=\sum_{j=1}^d \sum_{x\in B\cap U_j} x$. Let $B_1,\ldots,B_h$ be the sets enumerated by $Q(a_1,\ldots,a_d,b_1,\ldots,b_d)$, ordered by decreasing weight (breaking ties arbitrarily). For each $B_i$, fix a permutation $\pi:[d]\to[d]$ so that $B_i\cap U_j$ is a subset of $[a_{\pi(j)}+b_j]$ of size $b_j$, and let $A_i=A_i(\pi)$ be the set satisfying for every $1\le j\le d$ that $A_i\cap U_j=[a_{\pi(j)}+b_j]\setminus (B\cap U_j)$.
The proof would follow by showing that $A_1,\ldots,A_h$ and $B_1,\ldots,B_h$ satisfy the conditions of Theorem~\ref{thm:Combin}.

It is clear from the definition of $A_i$ and $B_i$ (viewed as subsets of $U_1\cup\cdots\cup U_d$) that the third and fourth conditions in the statement are satisfied. As for the first condition, note that for every $1\le i\le h$ we have (by definition) $A_i\cap B_{i'}=\emptyset$. Hence, it remains to show that if $A_i\cap B_{i'}=\emptyset$ and $i\neq i'$ then $i'<i$. Fixing $i\neq {i'}$ for which $A_i\cap B_{i'}=\emptyset$, we will show that $w(B_{i'})>w(B_i)$, thus completing the proof.
Observe that for every $x\in B_{i'}\cap U_j$ either $x\in B_i\cap U_j$ or else $x > a_{\pi(j)}+b_j$ (i.e., $x\in U_j\setminus[a_{\pi(j)}+b_j]$), as otherwise $x \in A_i$. Since $\size{B_{i'}\cap U_j}=\size{B_i\cap U_j}$ ($=b_j$), it follows that $\sum_{x\in B_{i'}\cap U_j} x \ge \sum_{x\in B_i\cap U_j} x$, and moreover, this inequality is strict if $B_{i'}\cap U_j \neq B_i\cap U_j$. Since $B_{i'} \neq B_i$ there must be at least one $j$ for which $B_{i'}\cap U_j \neq B_i\cap U_j$, implying that $w(B_{i'}) > w(B_i)$, as desired.
\end{proof}

\subsection{Explicit formula for $Q(a_1,\ldots,a_d,b_1,\ldots,b_d)$}

For a permutation $\pi:[d]\to[d]$, let $\F_{\pi}$ be the family of sets $B\subseteq U_1\cup\cdots\cup U_d$ satisfying $B\cap U_{i}\subseteq [a_{\pi(i)}+b_i]$ and $\size{B\cap U_i}=b_{i}$.
Denoting $\S_d$ the set of permutations on $[d]$, we have the inclusion-exclusion formula
\begin{equation}\label{eq:Qformula1}
Q(a_1,\ldots,a_d,b_1,\ldots,b_d) = \size{\bigcup_{\pi\in \S_d} \F_{\pi}} = \sum_{\emptyset\neq I\subseteq \S_d}
(-1)^{\size{I}+1}\size{\bigcap_{\pi\in I} \F_{\pi}}\;.
\end{equation}
Notice that for every $\pi\in \S_d$ we have $\size{\F_{\pi}} = \prod_{i=1}^d \binom{a_{\pi(i)}+b_i}{b_i}$. More generally, for any $\emptyset\neq I\subseteq \S_d$, putting $a^I_i=\min_{\pi\in I} a_{\pi(i)}$ we clearly have
\begin{equation}\label{eq:Qformula2}
\size{\bigcap_{\pi\in I} \F_{\pi}} = \prod_{i=1}^d \binom{a^I_i+b_{i}}{b_i} \;.
\end{equation}
Plugging~(\ref{eq:Qformula2}) into~(\ref{eq:Qformula1}) gives an explicit formula for $Q(a_1,\ldots,a_d,b_1,\ldots,b_d)$.
As an example, we get for $d=2$ that if $a_1\le a_2$ then
$$Q(a_1,a_2,b_1,b_2) = \binom{a_1+b_1}{b_1}\binom{a_2+b_2}{b_2}+\binom{a_2+b_1}{b_1}\binom{a_1+b_2}{b_2}-\binom{a_1+b_1}{b_1}\binom{a_1+b_2}{b_2}\;.$$

\subsection{Alternative proof of Theorem~\ref{thm:Main}}\label{subsec:alternative}

At the beginning of Section~\ref{sec:Intro} we claimed that Theorem~\ref{thm:Combin} is the most general result of this paper. Let us briefly explain how can one derive the lower bound part of Theorem~\ref{thm:Main} from Theorem~\ref{thm:Combin}.
Let $G$ be a $d$-uniform $d$-partite hypergraph that is $K^d_{p_1,\ldots,p_d}$-saturated, and suppose $e_1,\ldots,e_h$ is a corresponding saturation process, that is, an ordering of the non-edges of $G$ such that, after all edges $e_{i'}$ with $i'<i$ are added to $G$, adding $e_i$  creates a new copy $C_i$ of $K^d_{p_1,\ldots,p_d}$.
Note that for each edge $e\in C_i$, either $e\in G$ or else $e=e_{i'}$ for some $i'\le i$.
Put $A_i=V(G)\setminus V(C_i)$ and observe that $A_1,\ldots,A_h$ and $e_1,\ldots,e_h$ satisfy the first and second conditions of Theorem~\ref{thm:Combin}; indeed, we have $e_i\subseteq V(C_i)$ so $A_i\cap e_i=\emptyset$, while for $i<j$ we have $e_j\nsubseteq V(C_i)$ so $A_i\cap e_j\neq\emptyset$.
Now, denote $V_1,\ldots,V_d$ the vertex classes of $G$, and suppose all are of size $n$.
Since $C_i$ is a copy of $K^d_{p_1,\ldots,p_d}$, there is a permutation $\pi:[d]\to[d]$ so that $\size{C_i\cap V_j}=p_{\pi(j)}$.
It follows that for every $1\le i\le h$ there is a permutation $\pi$ so that $\size{A_i\cap V_j}=n-p_{\pi(j)}$; moreover, clearly $\size{e_i\cap V_j}=1$.
Therefore, by applying Theorem~\ref{thm:Combin} we
deduce that the number of edges of $G$, which is $n^d-h$, is at least $n^d-Q(n-p_1,\ldots,n-p_d,1,\ldots,1)$.
This can be used to give an alternative proof of the lower bound part of Theorem~\ref{thm:Main} since it can be shown that $Q(n-p_1,\ldots,n-p_d,1,\ldots,1)=q_n(p_1,\ldots,p_d)$.

\section{Concluding Remarks and Open Problems}\label{sec:conclude}

\paragraph*{Undirected strong saturation in bipartite graphs:}
Recall that our main result (specialized to graphs) shows that in the setting of weak saturation, the undirected version $W_n(p,q)$ requires much fewer edges
than its directed analogue $\WL$. It is thus natural to ask what happens in the setting of strong saturation. As we now discuss, we conjecture
that while the undirected version is easier than the directed one, it is only easier by an additive constant factor.

Let $S_n(p,q)$ be the minimum number of edges in an $n\times n$ bipartite graph such that any addition of a new edge between its two classes creates a copy of $K_{p,q}$ (i.e., the graph is strongly $K_{p,q}$-saturated).  Let $\overrightarrow{S_n}(p,q)$ denote the \emph{directed} analogue\footnote{I.e., where the copies of $K_{p,q}$ must have their $p$ vertices in the first class and their $q$ vertices in the second class.} of $S_n(p,q)$.
Answering a conjecture of Erd{\H{o}}s-Hajnal-Moon~\cite{ErdosHaMo64}, $\overrightarrow{S_n}(p,q)$ was completely determined by Wessel~\cite{Wessel66} and Bollob\'as~\cite{Bollobas67} to be $(p+q-2)n-(p-1)(q-1)$.\footnote{In fact, this is a special case of Alon's Theorem~\ref{thm:Noga}; see~\cite{alon}.}
Perhaps surprisingly, there are constructions showing that $S_n(p,q)$ is, in general, strictly smaller than this (of course, there is no distinction between the undirected and directed versions when $p=q$). To see this, suppose $p \le q$ and let $G^k_{p,q}$ be any $n\times n$ bipartite graph having $p-1$ vertices in each class complete to the other class, some $k$ additional vertices in each class spanning a $K_{k,k}$, and where the remaining vertices have degree $q-1$.
Note that $G^k_{p,q}$ has the property that any new edge one adds to it has an endpoint of degree at least $q$. The $q$ neighbors, together with the $p$ complete vertices from the other class, then form a $K_{p,q}$, implying that $G^k_{p,q}$ is strongly $K_{p,q}$-saturated. One can check that $G^k_{p,q}$ has in fact $\overrightarrow{S_n}(p,q)-k(q-p-k)$ edges. Optimizing using $k=\floor{(q-p)/2}$ gives $S_{n}(p,q) \le \overrightarrow{S}_{n}(p,q) - \floor{(q-p)^2/4}$, for every $n$ large enough such that $G^k_{p,q}$ is well defined. 
We conjecture that this upper bound is best possible for large enough $n$, that is, the ``gain'' over the directed version is an additive constant.

\begin{conj}\label{conj:saturaion}
For every $p,q$ there is an integer $n_0$ such that for every $n\ge n_0$ we have
\begin{equation}\label{eqConj}
S_{n}(p,q) = \overrightarrow{S}_{n}(p,q) - \floor{\frac{(q-p)^2}{4}}\;.
\end{equation}
\end{conj}

The only case for which we can confirm the above conjecture is when $p=1$ and $q$ is arbitrary. 
Consider any $n\times n$ bipartite graph that is strongly $K_{1,q}$-saturated, and observe that the set of vertices of degree strictly smaller than $q-1$ must span a clique. Consider the class containing the fewest such low-degree vertices, and let $k$ be their number there. Summing the degrees of all vertices in that class, we get that the number of edges in the graph is at least
$$k^2+(n-k)(q-1)=(q-1)n-k(q-1-k) \ge (q-1)n-\floor{\frac{(q-1)^2}{4}} \;,$$
which agrees with (\ref{eqConj}).

We note that for general $2 \leq p < q$ we do not have any (non trivial) lower bound for $S_{n}(p,q)$. It will thus be interesting to prove even a weaker
version of the above conjecture, by establishing that 
$S_{n}(p,q) \geq \overrightarrow{S}_{n}(p,q) - C$ for some constant $C=C(p,q)$ that depends on $p$ and $q$ (and is independent of $n$).

\paragraph*{Variants of Theorem~\ref{thm:Combin}:}
It would be interesting to know what is the best possible bound one gets in  Theorem~\ref{thm:Combin} if, for example, one replaces the first and second conditions with a non-skew one, that is, $A_i\cap B_j=\emptyset$ if and only if $i=j$, as in Bollob\'as's Two Families Theorem. Such a variant would have implications for strong saturation; indeed, a special case of this was described in the previous item.
Note that the bound in Alon's Two Families Theorem cannot be improved if one replaces the skew condition with a non-skew one (as the natural extremal construction satisfies the non-skew condition as well).
Interestingly, it can be shown that the bound in Theorem~\ref{thm:Combin} is generally not best possible if one requires the non-skew condition instead (see the construction in the previous item).
Nevertheless, it seems reasonable to conjecture that the correct bound in the non-skew case with, e.g., $b_i=b$ constant should not in general be much larger than the bound $\prod_{i=1}^d \binom{a_i+b}{b}$ implied by Alon's theorem.

It would also be interesting to consider variants of the third and fourth conditions of Theorem~\ref{thm:Combin} that are symmetric with respect to $a_i$ and $b_i$. For example, one might instead require that there be a permutation $\pi:[d]\to[d]$ so that both $\size{A_i\cap X_j}\leq a_{\pi(j)}$ and $\size{B_i\cap X_j}\leq b_{\pi(j)}$. As another example, we may instead require the existence of two permutations $\pi,\sigma:[d]\to[d]$ such that $\size{A_i\cap X_j}\leq a_{\pi(j)}$ and $\size{B_i\cap X_j}\leq b_{\sigma(j)}$.
Notice that these two variants would have implications for weak saturation of classes of hypergraphs more general than $d$-uniform $d$-partite hypergraphs (this can be easily seen by following the reduction at the end of Section~\ref{sec:thm_combin}).

\paragraph{$H$-free saturation:}
Notice that when defining whether a hypergraph $G$ is strongly $H$-saturated, one may or may not require that $G$ be $H$-free.
Indeed, some authors make this requirement (e.g.,~\cite{Tuza84}) and some do not (e.g.,~\cite{alon}).
It would be interesting to know in this regard whether there is some $H$ for which requiring $H$-freeness changes the corresponding strong saturation number.

\end{document}